\documentclass[pdflatex,sn-mathphys-num]{sn-jnl}

\usepackage{amsthm,amsmath,amsbsy,amssymb,amsfonts,mathabx}

\usepackage{stackrel}
\usepackage[bbgreekl]{mathbbol}
\DeclareSymbolFontAlphabet{\mathbb}{AMSb}
\DeclareSymbolFontAlphabet{\mathbbl}{bbold}
\usepackage{mathrsfs}

\usepackage[dvipsnames]{xcolor}

\usepackage[shortlabels]{enumitem}

\usepackage{graphicx,placeins}

\newtheorem{conj}{Conjecture}

\newtheorem{defi}[conj]{\bf Definition}

\newtheorem{prop}[conj]{\bf Proposition}

\newtheorem{rem}[conj]{\bf Remark}
\newtheorem{example}[conj]{\bf Example}

\newtheorem{assumpt}[conj]{\bf Assumption}

\makeatletter
\def\saveenum{\xdef\@savedenum{\the\c@enumi\relax}}
\def\resetenum{\global\c@enumi\@savedenum}
\makeatother

\def\to{\rightarrow}

\def\sign{{\rm sign}}

\def\Bc{{\mathcal B}}
\def\Cc{{\mathcal C}}

\def\Fc{{\mathcal F}}
\def\Gc{{\mathcal G}}
\def\Hc{{\mathcal H}}

\def\Lc{{\mathcal L}}

\def\Tc{{\mathcal T}}

\def\Vc{{\mathcal V}}

\def\Zc{{\mathcal Z}}

\def\Rb{{\mathbb R}}

\def\Pb{{\mathbb P}}

\def\Cscr{{ \mathscr C}}
\def\Gscr{{ \mathscr G}}
\def\Vscr{{ \mathscr V}}

\def\contiguous{\triangleleft\kern-.20em\triangleright}

\def\implies{\Longrightarrow}

\def\ii{{\mathbf 1}( }
\def\0{{\mathbf 0} }
\def\1{{\mathbf 1} }
\def\x{ {\bf x}}

\def\x{ {\bf x}}
\def\y{ {\bf y}}
\def\z{ {\bf z}}

\def\u{ {\bf u}}

\def\T{ {\bf T}}

\def\X{ {\bf X}}
\def\Y{ {\bf Y}}
\def\Z{ {\bf Z}}

\newcommand{\Card}{{\mathrm{Card}}}

\newcommand{\SetCopulas}{{\Cscr}}
\newcommand{\SetCopulasCond}{{\SetCopulas_{\textnormal{cond}}}}

\begin{document}

\title[Measures of non-simplifyingness
for conditional copulas and vines]{Measures of non-simplifyingness for conditional copulas and vines}


\author*{\fnm{Alexis} \sur{Derumigny}}\email{a.f.f.derumigny@tudelft.nl}

\affil{\orgdiv{Department of Applied Mathematics}, \orgname{Delft University of Technology}, \orgaddress{\street{Mekelweg 4}, \postcode{2628 CD}, \city{Delft}, \country{Netherlands}}}


\abstract{
In copula modeling, the simplifying assumption has recently been the object of much interest.
Although it is very useful to reduce the computational burden, it remains far from obvious whether it is actually satisfied in practice.
We propose a theoretical framework which aims at giving a precise meaning to the following question: how non-simplified or close to be simplified is a given conditional copula?
For this, we propose a new framework centered at the notion of measure of non-constantness.
Then we discuss generalizations of the simplifying assumption to the case where the conditional marginal distributions may not be continuous, and corresponding measures of non-simplifyingness in this case.
The simplifying assumption is of particular importance for vine copula models, and we therefore propose a notion of measure of non-simplifyingness of a given copula for a particular vine structure, as well as different scores measuring how non-simplified such a vine decompositions would be for a general vine.
Finally, we propose estimators for these measures of non-simplifyingness given an observed dataset.
A small simulation study shows the performance of a few estimators of these measures of non-simplifyingness.
}



\maketitle

\section{Introduction}

In conditional copula modeling, the simplifying assumption has recently been the object of much interest, see e.g. \cite{DerumignyFermanian2017,mroz2021simplifying,nagler2024simplified} and references therein.
We consider two random vectors of interest $\X = (X_1, \dots, X_d)$ and $\Z = (Z_1, \dots, Z_p)$.
By Sklar's theorem, the conditional joint cumulative distribution function $F_{\X | \Z}(\, \cdot \, | \z)$ of $\X$ given $\Z = \z$ can be decomposed as 
\begin{align*}
    F_{\X | \Z}(\x | \z) = C_{\X | \Z}
    \Big( F_{X_1 | \Z} (x_1 | \z), \dots,
    F_{X_d | \Z} (x_d | \z)
    \, \Big| \, \Z = \z \Big),
\end{align*}
for every $\z \in \Rb^p$,
where $F_{X_1 | \Z}, \dots, F_{X_d | \Z}$ are the conditional marginal cumulative distribution functions of respectively $X_1, \dots, X_d$ given $\Z$.
The simplifying assumption corresponds to the statement that the conditional copula
$C_{\X | \Z}( \, \cdot \, , \dots, \, \cdot \, | \, \Z = \z)$
does not depend on the value $\z$ of the conditioning vector.

\medskip

Although the simplifying assumption is very useful to reduce the computational burden\footnote{since the statistician only needs to estimate \textit{one} copula instead of an \textit{infinite amount} of copulas. Indeed, if the simplifying assumption is not satisfied, the statistician needs to specify and estimate a potentially different copula for each and every value $\z$ of the conditioning variable $\Z$.},
it remains far from obvious whether it is actually satisfied in practice.
As some authors -- see e.g. \cite{spanhel2019simplified,nagler2024simplified} mention, it is unrealistic to imagine that the simplifying assumption would be satisfied strictly speaking.
Nevertheless, if the simplifying assumption is somehow ``close to be satisfied but not exactly'', it may still be useful to assume it.
From a theoretical point-of-view, it then becomes necessary to define what the previous sentence really means.
How can we define what ``close to be simplified'' rigorously means, in mathematical terms?
The goal of this paper is to answer this question, by proposing a new concept of \emph{measure of non-simplifyingness}.

\medskip

Tests of the simplifying assumptions have already been developed, see e.g. \cite{acar2013statistical,DerumignyFermanian2017,gijbels2017nonparametric,gijbels2017score,kurz2022testing}, but they are very strict and, for a sample size large enough, they will detect any deviation from the simplifying assumption no matter how small it is.
This is classical in mathematical statistics: in usual situations, the power of a test will tend to $1$ under any fixed alternative.

\medskip

This paper starts by introducing a more general concept of ``measure of non-constantness'' (Section~\ref{sec:measure_nonconstantness}).
These are operators that measure how non-constant a function is. 
In a similar way, we present the new concept of ``measure of non-simplifyingness'' for conditional copulas in Section~\ref{sec:measure_nonsimplifyingness_condcop}.
In Section~\ref{sec:measure_nonsimplifyingness_vines}, we present extensions to vine copula models, to define non-simplifyingness scores.
Statistical inference of all these measures is discussed in Section~\ref{sec:statistical_inference}.
A small simulation study is presented in Section~\ref{sec:simulation_study}.
A list of all the proposed measures and their estimators is given in Appendix~\ref{sec:list_new_measures}.

\bigskip

\noindent
\textbf{Notation.} $\Card(A)$ denotes the cardinal of a set $A$. For two sets $A$ and $B$, we denote by $\Fc(A, B)$ the set of functions from $A$ to $B$.

\section{Measures of non-constantness}
\label{sec:measure_nonconstantness}

Let $\Zc$ be a set, let $E$ be a real vector space, and let $\Gc$ be a subset of $\Fc(\Zc, E)$, the set of functions from $\Zc$ to $E$.

\begin{defi}
    We say that a function $\psi: \, \Gc \to [0, +\infty]$
    is a \emph{measure of non-constantness} if it satisfies the following conditions:
    \begin{enumerate}[(i)]
        \item \emph{[Identification of constant functions]} For any function $f \in \Gc$, $\psi(f) = 0$ if and only if $f$ is a constant function.

        \item \emph{[Invariance by translation]}
        For any function $f \in \Gc$, for any constant $e \in E$ such that $f + e \in \Gc$,
        $\psi(f + e) = \psi(f)$, where $f + e$ denotes the function $x \mapsto f(x) + e$.

        \item \emph{[Sub-additivity]}
        For any functions $f, g \in \Gc$ such that 
        $f + g \in \Gc$,
        $\psi(f + g) \leq \psi(f) + \psi(g)$.

        \item  \emph{[Homogeneity]}
        For any function $f \in \Gc$, for any real $a \in \Rb$ such that 
        $a \times f \in \Gc$, we have
        $\psi(af) = |a| \times \psi(f)$, with the convention $0 \times \infty = 0$.
    \end{enumerate}
    A function $\psi$ satisfying (ii), (iii), (iv) and
    \begin{itemize}
        \item[(i')] For any constant function $f \in \Gc$, $\psi(f) = 0$.
    \end{itemize}
    is called a \emph{pseudo-measure of non-constantness}.
\end{defi}

Axiom \emph{(i)} is natural in the sense that a measure of non-constantness should be $0$ when the function $f$ is constant (because then there are no variations of $f$). Ideally, this should be the case only when $f$ is constant, but this may be too constraining sometimes. By analogy with the concepts of norm and pseudo-norm, we give a relaxed version \emph{(i')} of \emph{(i)}, and call the corresponding object a pseudo-measure of non-constantness.

\medskip

Axiom \emph{(ii)} means that the measure is invariant upon addition of a constant, since this should not change the way the function $f$ is non-constant.
The two last axioms \emph{(iii)} and \emph{(iv)} are inspired from the definition of a norm.
Indeed, the function $f + g$ should not vary more than both $f$ and $g$, considered separately.
Finally, multiplying a function $f$ by a constant factor should only have a multiplicative effect on the measure of non-constantness of $f$.

\medskip

A first natural idea to construct measures of non-constantness is to rely on the norm of the non-constant part of a function.
This is detailed in the following example.

\begin{example}
\label{ex:measure-non-const_pseudo-norm}
    Let $\textnormal{Const}$ be the set of constant functions from $\Zc$ to $E$. Then $\textnormal{Const}$ is a subspace of $\Fc(\Zc, E)$. Let $\widetilde\Gc$ be a space of $\Fc(\Zc, E)$ linearly independent of $\textnormal{Const}$, and let $\Gc := \textnormal{Const} \oplus \widetilde\Gc$.
    Then every (pseudo-)norm $\| \cdot \|_{\widetilde\Gc}$ on $\widetilde\Gc$ induces a (pseudo-)measure of non-constantness on $\Gc$ by $\psi(\widetilde g + c) := \| \widetilde g \|_{\widetilde\Gc}$.
\end{example}

\begin{example}
\label{ex:measure-non-const_trivial}
    The discrete map $f \mapsto \1\{f \notin \textnormal{Const}\}$ is always a measure of non-constantness, but it is the least useful since it assigns $1$ to all non-constant functions without any distinction.
\end{example}

We now give several more applicable examples of ways on how to construct $\text{(pseudo-)measures}$ of non-constantness in the case where the vector space $E$ is equipped with a pseudo-norm $\| \cdot \|_E$.

\begin{example}
\label{ex:measure-non-const_Kolmogorov-Smirnov}
    First, we can define the Kolmogorov-Smirnov pseudo-measure of non-constantness by
    \begin{align*}
        \psi_{\text{KS}}(f)
        := \sup_{x, y \in \Zc} \| f(x) - f(y) \|_E.
    \end{align*}
    This is a measure of non-constantness whenever the pseudo-norm $\| \cdot \|_E$ is actually a norm.
    Moreover, fixing a given collection $z_1, \dots, z_n \in \Zc$,
    one can define other pseudo-measures of non-constantness such as $\sup_{i} \| f(z_i) - f(z_{i+1}) \|_E$,
    $\sup_{i,j} \| f(z_i) - f(z_j) \|_E$,
    or corresponding sum-type measures
    $\Sigma_{i} \| f(z_i) - f(z_{i+1}) \|_E$,
    $\Sigma_{i,j} \| f(z_i) - f(z_j) \|_E$.
    These will only be pseudo-measures of non-constantness, not measures of non-constantness
    (unless $\Zc = \{z_1, \dots, z_n\}$)
    but they are straightforward to implement.
\end{example}

\begin{example}
\label{ex:measure-non-const_integral-type}
    If $(\Zc, \Bc(\Zc), \mu)$ is a measured space in the context of the previous example, integral-type pseudo-measures of non-constantness become available.
    They can be defined as 
    \begin{align*}
        \psi(f)
        := \left( \iint \| f(x) - f(y) \|_E^s d\mu(x) d\mu(y) \right)^{1/s},
    \end{align*}
    for $s \in (1, +\infty)$.
    To avoid the double integral, it can be easier to fix a collection $z_1, \dots, z_n \in \Zc$, and to use instead the pseudo-measure
    \begin{align*}
        \psi(f)
        := \bigg( \sum_i
        \int \|f(z_i) - f(x) \|_E^s d\mu(x) \bigg)^{1/s}.
    \end{align*}
\end{example}

\begin{example}
\label{ex:measure-non-const_from-averaging}
    In many cases, there exist an averaging operator
    $\textnormal{ave}: \Gc \to E$ such that
    \begin{enumerate}[(i)]
        \item The mapping $\textnormal{ave}$ is linear.
        \item If $f$ is constant with a certain value $e \in E$, 
        then $\textnormal{ave}(f) = e$.
    \end{enumerate}
    For example $\textnormal{ave}(f) = \int f(z) d\mu(z)$ for a probability measure $\mu$ satisfies these conditions.
    If an averaging operator is available,
    pseudo-measures of non-constantness can be defined using the norm of the difference between $f$ and its average, by $\sup_z \| f(z) - \textnormal{ave}(f) \|_E$ or
    $\int_z \| f(z) - \textnormal{ave}(f) \|_E \, d\mu(z)$.
\end{example}

\begin{rem}
    All the previous examples can be generalized to pseudo-metrics $d_E$ which satisfy the translation-invariance condition $d_E(f+e, g+e) = d_E(f,g)$ for every $f,g \in \Gc, e \in E$ such that $f+e, g+e \in E$.
\end{rem}

\begin{rem}
\label{rem:structure_set_measures_non-constantness}
    Note that the set of measures of non-constantness is a convex cone: if $\psi_1$ and $\psi_2$ are measures of non-constantness, and $\alpha_1 > 0, \alpha_2 \geq 0$, then $\alpha_1 \psi_1 + \alpha_2 \psi_2$ is also a measure of non-constantness. Similarly, the set of pseudo-measures of non-constantness is a pointed convex cone (since it contains the zero function $\psi_0: f \mapsto 0$).
    As a consequence, new measures of non-constantness can be created by weighted combinations of existing measures of non-constantness. This can be useful in practice to combine different ways of measuring non-constantness together.
\end{rem}

It seems coherent that a measure of non-constantness of a function $f$ could be linked to the derivative of $f$.
We first present a general framework before defining the corresponding measure of non-constantness.
Let us assume that $\Zc$ is a connected open set in a linear topological set, and that $E$ is a locally convex linear topological space.
Let $\Gc$ be the space of Gâteaux-differentiable functions from $\Zc$ to $E$. 
Recall \cite{averbukh1967theory} that a function
$f: \, \Zc \to E$ is Gâteaux-differentiable
if there exists a linear operator $A =: f'(x)$ such that for
$x, h \in \Zc$,
$f(x + h) = f(x) + Ah + r(h)$,
with $r(th) / t \to 0$ for every $h$.
We know that $f \in \Gc$ is constant if and only if its Gâteaux-derivative is equal to zero at each point of $\Zc$, see e.g. \cite[Theorem 1.9 page 219]{averbukh1967theory}.

\begin{example}[Measures of non-constantness from derivatives]
    Therefore, $\psi(f) := \| f' \|$ is a measure of non-constantness, where $\| \, \cdot \, \|$ is a norm on the space $\Fc(\Zc, \Lc(\Zc, E)))$ of maps from $\Zc$ to the space $\Lc(\Zc, E)$ of linear operators from $\Zc$ to $E$.
    For example, if $\Zc \subset \Rb$ and $E = \Rb$, then $\| f' \|$ could be chosen as $\sup_{z \in \Zc} | f'(z) |$ or $\int_{z \in \Zc} | f'(z) | d\mu(z)$.
\label{example:measure_nonconstantness_derivatives}
\end{example}

\section{Measures of non-simplifyingness for conditional copulas}
\label{sec:measure_nonsimplifyingness_condcop}

\subsection{Framework}

Remember that $\X$ and $\Z$ are two random vectors, of respective dimensions $d$ and $p$.
Let us denote by $\SetCopulas$ the set of all $d$-dimensional copulas.
For a given point $\z \in \Rb^p$, one can define the conditional joint cumulative distribution function $F_{\X | \Z}( \, \cdot \, | \Z = \z)$ of $\X$ given $\Z = \z$.
Note that the margins of $F_{\X | \Z}( \, \cdot \, | \Z = \z)$ are the conditional cumulative distribution function $F_{X_1 | \Z}( \, \cdot \, | \Z = \z), \dots, F_{X_d | \Z}( \, \cdot \, | \Z = \z)$.

\medskip

\begin{assumpt}
    $\forall i = 1, \dots, d$, the function $F_{X_i | \Z}( \, \cdot \, | \Z = \z)$ is continuous.
\label{assumpt:continuous_cond_cdf}
\end{assumpt}

Under Assumption~\ref{assumpt:continuous_cond_cdf},
the conditional copula $C_{\X | \Z}(\, \cdot \, | \Z = \z)$
of $\X$ given $\Z = \z$ is unique and is given by
\begin{align}
    C_{\X | \Z}
    \Big(u_1, \dots, u_d \, | \, \Z = \z \Big)
    = F_{\X | \Z}
    \Big( F^-_{X_1 | \Z} (u_1 | \z), \dots,
    F^-_{X_d | \Z} (u_d | \z)
    \, \Big| \, \Z = \z \Big),
\label{eq:def:cond_copula}
\end{align}
for every $\u = (u_1, \dots, u_d) \in [0,1]^d$,
where $F^-$ represents the (generalized) inverse of a (univariate) cumulative distribution function $F$.
Note that Assumption~\ref{assumpt:continuous_cond_cdf} is not the same as assuming that all the margins $F_1, \dots, F_d$ of $\X$ are continuous.
In the following, we will assume that Assumption~\ref{assumpt:continuous_cond_cdf} is satisfied for all $\z \in \Rb^p$, except in Section~\ref{subsec:non-continuous_cond_cdf} which precisely studies what happens when Assumption~\ref{assumpt:continuous_cond_cdf} is not satisfied.

\medskip

To be precise, we denote by $\SetCopulasCond := \Fc(\Rb^p, \SetCopulas)$ the set of conditional copulas, i.e. $\SetCopulasCond$ is the set of all (measurable) functions from $\Rb^p$ to the set $\SetCopulas$ of all copulas.
It is also possible to fix a distribution $\Pb_\Z$ on $\Rb^p$ and to consider the quotient space $\SetCopulasCond/\Pb_\Z$ where equality of conditional copulas is only considered $\Pb_\Z$-almost everywhere.

\medskip

Indeed, the simplifying assumption can be interpreted in two different ways, depending on whether the mapping $\z \mapsto C_{\X | \Z = \z}$ is assumed to be constant, or only $\Pb_\Z$-almost surely constant.
This is related to the fact that the conditional joint cumulative distribution function $F_{\X | \Z}( \, \cdot \, | \Z = \z)$ is itself uniquely defined only $\Pb_\Z$-almost surely.
Of course in practice this does not make a difference, but for the theory this means that the measure of non-simplifyingness could depend on the law $\Pb_\Z$.
This could be considered as an advantage:
we can take into account potential non-uniformities of $\Pb_\Z$ since some values $\z$ may happen often more than others.
But this could also be seen as a drawback: since we need to know the true law $\Pb_\Z$ -- which is typically not the case in practice -- or rely on an estimate thereof.

\begin{example}
    To illustrate the impact, let us consider $\Zc = [-1,1]$, $d = 2$, and the conditional copula
    $C_{\X | Z}(\u | z) := \text{GaussianCopula}_{\rho = 0.8 z^2} (\u)$.
    If $Z$ is uniform on $[-1, 1]$, then the simplifying assumption (for the conditional copula $C_{\X | Z}$) is not satisfied.
    But if $Z$ is uniform on $\{-1, 1\}$ instead, then the simplifying assumption is satisfied, because $Z$ put all its mass on two points, at which the conditional copulas are identical.
    This shows that the simplifying assumption depends, not only on the conditional copula as a function $\Zc \mapsto \SetCopulas$, but also on the choice of the measure $\Pb_\Z$.
    This is reflected in the two definitions that are presented below.
\end{example}

We will now define the concept of measure of non-simplifyingness.
For this, we will need the following notation.
For $k \geq 1$, let $\mathfrak{S}_k$ be the set of permutations of $\{1, \dots, k\}$. For $\pi \in \mathfrak{S}_k$
and $\x \in \Rb^k$, we denote by $\pi(\x)$ the permuted vector 
$(x_{\pi(i)})_{i = 1, \dots, k}$.

\begin{defi}
    We say that a function $\psi: \SetCopulasCond \to [0, +\infty]$ (respectively $\psi: \SetCopulasCond/\Pb_\Z \to [0, +\infty]$) is
    \emph{a measure of non-simplifyingness}
    (respectively a $\Pb_\Z$\emph{-measure of non-simplifyingness})
    if it satisfies the following conditions:
    \begin{enumerate}
        \item[(i)] [Identification of simplified copulas]
        For every $C \in \SetCopulasCond$ 
        (respectively $\SetCopulasCond/\Pb_\Z$), we have 
        $\psi(C) = 0$ if and only if $C$ satisfies the simplifying assumption.

        \smallskip

        \item[(ii)] [Invariance by permutation of components of $\X$ and $\Z$]
        We have $\psi(C_{\pi_\X(\X) | \pi_\Z(\Z)}) = \psi(C_{\X | \Z})$.
    \end{enumerate}
    A function $\psi$ satisfying (ii) and
    \begin{enumerate}
        \item[(i')] $\psi = 0$ if the simplifying assumption is satisfied.
    \end{enumerate}
    is called a \emph{pseudo-measure of non-simplifyingness}
    (respectively a $\Pb_\Z$\emph{-pseudo-measure of non-simplifyingness}).
\label{def:measure_nonsimplifyingness}
\end{defi}

Axiom \emph{(i)} is quite straightforward, as we want the measure of non-simplifyingness to take the value $0$ if and only if the copula is indeed simplified. Sometimes this is a bit too strict (for example, for measure of non-simplifyingness based only on conditional Kendall's tau or Spearman's rho) and this gives rise to pseudo-measures of non-simplifyingness instead. This justifies the existence of Axiom \emph{(i')}.

\medskip

Axiom \emph{(ii)} is also coherent with our intuitive understanding that the conditional copula $C_{(X_1, X_2) | (Z_1, Z_2)}$ is as simplified or as non-simplified as the conditional copulas $C_{(X_2, X_1) | (Z_1, Z_2)}$ or $C_{(X_2, X_1) | (Z_2, Z_1)}$. Note that these conditional copulas are different in general because we have not assumed that the random vectors are exchangeable.

\begin{rem}
    A similar comment can be made on the structure of the set of all measures of (pseudo)-non-simplifyingness as was done in Remark~\ref{rem:structure_set_measures_non-constantness}.
    Indeed, we can see that the set of ($\Pb_\Z$-)measures of non-simplifyingness is a convex cone, and that the set of  ($\Pb_\Z$-)pseudo-measures of non-simplifyingness is a pointed convex cone.
\end{rem}

Since a measure of non-simplifyingness depends only the conditional copula, it is invariant by marginal transformations, and even by conditional marginal transformations. This is formalized in the next result, whose proof mainly relies on the invariance principle for (unconditional) copulas (see \cite[Theorem 2.4.7]{hofert2018elements}).

\begin{prop}
    Let $g_1, \dots, g_d$ be functions from $\Rb \times \Rb^p$ to $\Rb$ that are strictly increasing with respect to their first argument.
    For $i = 1, \dots, d$, let $Y_i := g(X_i, \Z)$; let $\Y := (Y_1, \dots, Y_d)$.
    Then $C_{\Y | \Z} = C_{\X | \Z}$
    and therefore $\psi(C_{\Y | \Z}) = \psi(C_{\X | \Z})$ for any pseudo-measure of non-simplifyingness $\psi$.
\end{prop}

\begin{proof}[Proof.]
    For a $d$-dimensional cumulative distribution function $F$ with continuous margins, we denote by $\textnormal{Copula}(F)$ its copula. By extension, for a $d$-dimensional random vector $\T = (T_1, \dots, T_d)$ with continuous margins, we denote by $\textnormal{Copula}(\T)$ the copula of its joint cumulative distribution function.

    \medskip
    
    Let us fix $\u \in \Rb^d$ and $\z \in \Rb^p$.
    Let $\T = (T_1, \dots, T_d)$ be a random vector
    following the distribution $F_{\X | \Z = \z}$.
    Then
    \begin{align}
        F_{\Y | \Z = \z}(\y)
        &= \Pb(Y_1 \leq y_1, \dots, Y_d \leq y_d | \Z = \z) \nonumber \\
        &= \Pb(g(X_1, \Z) \leq y_1, \dots, g(X_d, \Z) \leq y_d | \Z = \z) 
        \nonumber \\
        &= \Pb(g(X_1, \z) \leq y_1, \dots, g(X_d, \z) \leq y_d | \Z = \z) 
        \nonumber \\
        &= \Pb(g(T_1, \z) \leq y_1, \dots, g(T_1, \z) \leq y_d).
        \label{eq:from_Y_to_T}
    \end{align}
    Then
    \begin{align*}
        C_{\X | \Z}(\u \, | \, \z)
        &= \textnormal{Copula}(F_{\X | \Z = \z})(\u) \\
        &= \textnormal{Copula}(T_1, \dots, T_d)(\u) \\
        &= \textnormal{Copula}(g(T_1, \z), \dots, g(T_d, \z))(\u) \\
        &= \textnormal{Copula}(F_{\Y | \Z = \z})(\u) \\
        &= C_{\Y | \Z}(\u \, | \, \z),
    \end{align*}
    The first and the last equalities are by definition of the conditional copula (Equation~\eqref{eq:def:cond_copula}),
    the second equality is by definition of $(T_1, \dots, T_d)$,
    the third equality is the application of the invariance principle \cite[Theorem 2.4.7]{hofert2018elements}.
    The fourth equality is a consequence of Equation~\eqref{eq:from_Y_to_T}: the cumulative distribution function of the random vector $(g(T_1, \z), \dots, g(T_d, \z))$ is $F_{\Y | \Z = \z}$.
\end{proof}

\subsection{Examples of measures and pseudo-measures of non-simplifyingness}

We now present several ways to construct measures of non-simplifyingness.
The first method is to apply the framework developed in the previous section, by recognizing that the space of conditional copulas $\SetCopulasCond$ is the space of function from $\Rb^p$ to the set $\SetCopulas$ of all copulas.

\begin{prop}
    Let $\psi$ be a measure of non-constantness on $\Fc(\Rb^p,\SetCopulas) = \SetCopulasCond$.
    We define a symmetrized version of $\psi$ by
    \begin{align*}
        \psi_{\text{sym}}(C_{\X | \Z}) := \frac{1}{d! \, p!}
        \sum_{\pi_\X \in \mathfrak{S}_d}
        \sum_{\pi_\Z \in \mathfrak{S}_p}
        \psi(C_{\pi_\X(\X) | \pi_\Z(\Z)}),
    \end{align*}
    for any $C_{\X | \Z} \in \SetCopulasCond$.
    Then $\psi_{\text{sym}}$ is a measure of non-simplifyingness.
\label{prop:measures_nonsimplifyingness-from-non-constantness}
\end{prop}

\begin{proof}[Proof.]
    A conditional copula $C_{\X | \Z}$ is simplified if and only
    $C_{\pi_\X(\X) | \pi_\Z(\Z)}$ is simplified for every permutation $\pi_\X, \pi_\Z$,
    if and only if $\psi(C_{\pi_\X(\X) | \pi_\Z(\Z)}) = 0$ for every permutation $\pi_\X, \pi_\Z$,
    if and only if $\psi_{\text{sym}}(C_{\X | \Z}) = 0$.
    This shows that $\psi_{\text{sym}}$ satisfies Axiom \emph{(i)} of Definition \ref{def:measure_nonsimplifyingness}.
    Let $(\pi_1, \pi_2) \in \mathfrak{S}_d \times \mathfrak{S}_p$.
    Then
    \begin{align*}
        \psi_{\text{sym}}(C_{\pi_1(\X) | \pi_2(\Z)}) 
        &= \frac{1}{d! \, p!}
        \sum_{\pi_\X \in \mathfrak{S}_d}
        \sum_{\pi_\Z \in \mathfrak{S}_p}
        \psi(C_{\pi_\X(\pi_1(\X)) | \pi_\Z(\pi_2(\Z))}) \\
        &= \frac{1}{d! \, p!}
        \sum_{\pi_\X \in \mathfrak{S}_d}
        \sum_{\pi_\Z \in \mathfrak{S}_p}
        \psi(C_{(\pi_\X \circ \pi_1)(\X) | (\pi_\Z \circ \pi_2)(\Z)}) \\
        &= \frac{1}{d! \, p!}
        \sum_{\pi_\X \in \mathfrak{S}_d}
        \sum_{\pi_\Z \in \mathfrak{S}_p}
        \psi(C_{\pi_\X(\X) | \pi_\Z(\Z)}) \\
        &= \psi_{\text{sym}}(C_{\X | \Z}),
    \end{align*}
    where the third equality is consequence of the fact that
    $(\mathfrak{S}_d, \circ)$ and $(\mathfrak{S}_p, \circ)$ are finite groups.
    This shows that $\psi_{\text{sym}}$ satisfies Axiom \emph{(ii)} of Definition \ref{def:measure_nonsimplifyingness}, as claimed.
\end{proof}

\medskip

Reusing the examples seen in Section~\ref{sec:measure_nonconstantness}, we obtain several measures of non-simplifyingness. Note that these measures are related to test statistics of the simplifying assumption obtained in \cite{DerumignyFermanian2017}. This is natural since the simplifying assumption is equivalent (by definition) to $\psi(C_{\X | \Z}) = 0$ for a measure of non-simplifyingness $\psi$.
There are many simple classes of such measures, for instance
\begin{align*}
    \psi
    = \| C_{\X | \Z = \cdot} - C_{\X | \Z , ave} \|
\end{align*}
for some norm $\| \cdot \|$, where an average conditional copula is given by
\begin{align*}
    C_{\X | \Z , ave}(\u)
    := \int C_{\X | \Z = \z}(\u) d\mu(\z),
\end{align*}
for any $\u \in [0, 1]^d$ and for some fixed probability measure $\mu$.
When $\mu = \Pb_\Z$, this average conditional copula becomes the partial copula, see~\cite[Proposition 1]{gijbels2015partial} and~\cite{bergsma2011nonparametric}.
More generally,
\begin{align}
    \psi
    = \| \phi(C_{\X | \Z = \cdot})
    - \phi(C_{\X | \Z , ave}) \|
\label{eq:dist-to-average-copula}
\end{align}
is a pseudo-measure of non-simplifyingness for a given mapping $\phi: \SetCopulas \to \Rb$.

\medskip

In particular, if $d = 2$, using Kendall's tau or Spearman's rho as the function $\phi$, we obtain pseudo-measures of non-simplifyingness such as
\begin{align}
    \psi
    = \| \tau_{\X | \Z = \cdot}
    - \tau_{\X | \Z , ave} \|,
\label{eq:dist-to-average-Kendall-tau}
\end{align}
or 
\begin{align}
    \psi
    = \| \rho_{\X | \Z = \cdot}
    - \rho_{\X | \Z , ave} \|,
\label{eq:dist-to-average-Spearman-rho}
\end{align}
where
$\tau_{\X | \Z , ave}$ denotes Kendall's tau of the average conditional copula $C_{\X | \Z , ave}$ and
$\rho_{\X | \Z , ave}$ denotes Spearman's rho of the average conditional copula $C_{\X | \Z , ave}$.
Alternatively, defining 
$\tau_{\X | \Z , ave} := \int \tau_{\X | \Z = \z} d\mu(z)$
and
$\rho_{\X | \Z , ave} := \int \rho_{\X | \Z = \z} d\mu(z)$
would also work, where $\tau_{\X | \Z = \z}$ and $\rho_{\X | \Z = \z}$ are
conditional Kendall's tau and conditional Spearman's rho of $\X$ given $\Z = \z$.
This can be extended to the case where $\X$ is of higher-dimension by using the matrix version of Kendall's tau or Spearman's rho.

\medskip

It is also possible to construct (pseudo-)measures of non-simplifyingness without needing averaging and the choice of a probability measure $\mu$.
Indeed, the mapping
\begin{align}
    \psi
    = \| (\u, \z, \z') \mapsto
    \phi(C_{\X | \Z = \z}(\u)) 
    - \phi(C_{\X | \Z = \z'}(\u)) \|.
\label{eq:comparison-copulas-pairwise}
\end{align}
is a measure of non-simplifyingness.
For example,
\begin{align}
    \psi
    = \sup_{(\u, \z, \z') \in [0,1]^d \times \Rb^p \times \Rb^p}
    \big| C_{\X | \Z = \z}(\u) - C_{\X | \Z = \z'}(\u) \big|,
\label{eq:comparison-copulas-pairwise-sup}
\end{align}
or
\begin{align}
    \psi
    = \Bigg( \int_{(\u, \z, \z') \in [0,1]^d \times \Rb^p \times \Rb^p}
    \big| C_{\X | \Z = \z}(\u) - C_{\X | \Z = \z'}(\u) \big|^s
    d\mu (\u, \, \z, \, \z') \Bigg)^{1/s},
\label{eq:comparison-copulas-pairwise-int}
\end{align}
for some measure $\mu$ on $[0,1]^d \times \Rb^p \times \Rb^p$ and $s \in (1, +\infty)$.
Note that these measures are more expensive to compute since they require the computation of a supremum or an integral over a potentially high-dimensional space.

\subsection{Measures of non-simplifyingness for particular sets of conditional copulas}

Often, we have information about the conditional copulas, for example by assuming a parametric or semi-parametric model.
Let us denote by $\Gscr$ a subset of the set $\SetCopulasCond$ of all conditional copulas. We say that a function $\psi: \, \Gscr \to [0, +\infty]$ is a measure of non-simplifyingness on $\Gscr$ if it satisfies Definition~\ref{def:measure_nonsimplifyingness} with $\SetCopulasCond$ replaced by $\Gscr$.

\begin{example}[Conditional copulas with densities]
    Let $\SetCopulas_{\text{dens}}$ be the set of all copulas that are absolutely continuous with respect to Lebesgue's measure.
    Let $\Gscr = \Fc(\Rb^p, \SetCopulas_{\text{dens}})$ be the set of conditional copulas $C_{\X | \Z}$ such that for all $\z \in \Rb^p$, $C_{\X | \Z = \z}$ has a (conditional) copula density $c_{\X | \Z = \z}$.
    The measures presented in the previous section can be adapted replacing conditional copulas by conditional copula densities.
    For example, one can consider
    \begin{align*}
    \psi = \| c_{\X | \Z = \cdot} - c_{\X | \Z , ave} \|
    \end{align*}
    or
    \begin{align*}
        \psi
        = \| (\u, \z, \z') \mapsto
        \phi(c_{\X | \Z = \z}(\u)) 
        - \phi(c_{\X | \Z = \z'}(\u)) \|.
    \end{align*}
\label{ex:measure-nonsimplified-densities}
\end{example}

Let $\{ C_\theta, \theta \in \Theta\}$ be a family of copulas.
Let us choose $\Gscr$ to be the set of conditional copulas of the form
$\z \in \Rb^p \mapsto C_{\theta(\z)}$, where $\theta: \, \Rb^p \to \Theta$.
We can then introduce measures of non-simplifyingness on $\Gscr$ based on a measure of non-constantness of the conditional parameter $\theta( \, \cdot \,)$, for example
\begin{align}
    \psi
    = \| \z \mapsto
    \theta(\z) - \theta_{ave} \|,
\label{eq:measure-nonsimplified-parametric}
\end{align}
for an average parameter $\theta_{ave} \in \Theta$, such as $\theta_{ave} = \int \theta(\z) \, d\mu(\z)$,
or
\begin{align}
    \psi
    = \| (\z, \z') \mapsto
    \theta(\z) - \theta(\z') \|.
\label{eq:measure-nonsimplified-parametric-pairwise}
\end{align}

Note that the parameter space $\Theta$ here needs not to be finite-dimensional.
In particular, if for every $\z \in \Zc$, the copula $C_{\X | \Z = \z}$ is the meta-elliptical copula
(see \cite{derumigny2022identifiability})
with conditional correlation matrix $\Sigma(\z)$ and conditional density generator $g_\z(\, \cdot \,)$, then a potential measure of non-simplifyingness is
\begin{align}
    \psi
    = \int \| \Sigma(\z) - \Sigma_{ave} \| d\z
    + \int \| g_\z(\, \cdot \,) - g_{ave} \| d\z,
\label{eq:measure-nonsimplified-elliptical}
\end{align}
given an average conditional correlation matrix $\Sigma_{ave}$ and an average generator $g_{ave}$, and appropriate choices of norms.
Similar definitions can be made for extreme value copulas, using a conditional version of the Pickands dependence function.

\subsection{Generalization to non-continuous conditional margins}
\label{subsec:non-continuous_cond_cdf}

Until now, we have only discussed the case where the conditional marginal distributions are all continuous (Assumption~\ref{assumpt:continuous_cond_cdf}); this ensures the uniqueness of the conditional copula $C_{\X | \Z}$ (in a $\Pb_\Z$-almost-sure sense).
We now discuss what can be done when Assumption~\ref{assumpt:continuous_cond_cdf} is no longer satisfied.
In this case, the conditional copula $C_{\X | \Z = \z}$ is uniquely determined only on $\text{Dom}_\z := \bigtimes_{i = 1}^d \textnormal{Ran} \big( F_{X_i | \Z = \z} \big)$.
Therefore, the simplifying assumption itself can be defined in several ways in this framework.

\medskip

We propose a first version of the simplifying assumption, which enforces that the conditional copulas are equal at every point $\u \in [0,1]^d$ for which both conditional copulas $C_{\X | \Z = \z}$ and $C_{\X | \Z = \z'}$ are uniquely defined.
Formally, this version of the simplifying assumption is
\begin{align*}
    \Hc_0^{\text{Dom}}: \,
    \forall \z, \z' \in \Zc^2, \,
    \forall \u \in \text{Dom}_\z \cap \text{Dom}_{\z'}, \,
    C_{\X | \Z = \z}(\u) = C_{\X | \Z = \z'}(\u).
\end{align*}

We now propose stricter generalizations of the simplifying assumption.
For every (joint) cumulative distribution function $F$, we denote by $\Cc(F)$ the set of copulas that are possible copulas of $F$.
We propose three other possible generalizations of the simplifying assumption using this concept.

\medskip

First, we could ask that the set of copulas corresponding to the distribution $F_{\X | \Z = \z}$ does not depend on $\z$. Formally, this means
\begin{align*}
    \Hc_0^{\text{equality}}: \,
    \forall \z, \z' \in \Zc^2, \,
    \Cc(F_{\X | \Z = \z}) = \Cc(F_{\X | \Z = \z'}).
\end{align*}
This may too strict to be useful. Indeed, if for some $\z$, the conditional marginal distributions of $\X | \Z = \z$ are continuous and for some other $\z'$, the conditional marginal distributions $\X | \Z = \z'$ are discrete, then $\Hc_0^{\text{equality}}$ will fail to hold.
Such phenomenon is contrary to the common intuition about copulas, which is that they should not depend or incorporate knowledge about the margins.

\medskip

Therefore, we propose a less strict version. We ask that for every two points $\z$ and $\z'$, there exists always (at least) one copula that can be the copula of $F_{\X | \Z = \z}$ and of $F_{\X | \Z = \z'}$.
Formally, this means
\begin{align*}
    \Hc_0^{\text{pairwise}}: \,
    \forall \z, \z' \in \Zc^2, \,
    \Cc(F_{\X | \Z = \z}) \cap \Cc(F_{\X | \Z = \z'}) \neq \emptyset.
\end{align*}

This intuition can be strengthen by asking that this copula is the same for every $\z$, leading to the assumption
\begin{align*}
    \Hc_0^{\text{intersection}}: \,
    \bigcap_{\z \in \Zc}
    \Cc(F_{\X | \Z = \z}) \neq \emptyset,
\end{align*}
that is, there exists a copula that works for all joint conditional cumulative distribution function $F_{\X | \Z = \z}$.
These generalizations are related together, as shown by the following result.
\begin{prop}
    The following implications hold:
    $$\Hc_0^{\text{equality}}
    \implies \Hc_0^{\text{intersection}}
    \implies \Hc_0^{\text{pairwise}}
    \implies \Hc_0^{\text{Dom}}.$$
\end{prop}
\begin{proof}[Proof.]
    The implication
    $\Hc_0^{\text{equality}} \implies \Hc_0^{\text{intersection}}$
    is direct: since all sets are equal and they are non-empty, then their intersection is not empty.
    The implication
    $\Hc_0^{\text{intersection}} \implies \Hc_0^{\text{pairwise}}$
    is also direct, since a non-empty joint intersection means that all pairwise intersections are not empty.
    The last implication is due to the fact that, if the two sets of copulas are equal, then the discontinuities of the conditional margins happen at the same points, and all the copulas in $\Cc(F_{\X | \Z = \z}) = \Cc(F_{\X | \Z = \z'})$ take the same values at those points.
\end{proof}

\medskip

From all these generalizations of the simplifying assumption to the non-continuous case, a corresponding notion of ``measure of non-simplifyingness'' can be defined by adapting Definition~\ref{def:measure_nonsimplifyingness} accordingly.
We remark that the assumption $\Hc_0^{\text{intersection}}$ seems to be the one that carries the most the intuition around the original simplifying assumption.

\medskip

We propose the following measure of non-simplfyingness, corresponding to $\Hc_0^{\text{pairwise}}$:
\begin{align}
    \psi = \sup_{\z, \z' \in \Zc} \,
    \inf_{C \in \Cc(F_{\X | \Z = \z}) \rule{0pt}{0.69em}} \,
    \inf_{C' \in \Cc(F_{\X | \Z = \z'}) \rule{0pt}{0.69em} }
    \| C - C' \|.
\label{eq:measure-nonsimplified-pairwise}
\end{align}
The intuition behind this expression is that we try to find the smallest distance between possible copulas to represent each of the two conditional distribution.

\medskip

On the contrary, if we follow $\Hc_0^{\text{equality}}$, we want both sets of conditional copulas to be exactly equal, and this motivates the definition of the following measure of non-simplifyingness:
\begin{align}
    \psi = \sup_{\z, \z' \in \Zc} \,
    \sup_{C \in \Cc(F_{\X | \Z = \z}) } \,
    \sup_{C' \in \Cc(F_{\X | \Z = \z'}) }
    \| C - C' \|.
\label{eq:measure-nonsimplified-strict}
\end{align}

\section{Measures of non-simplifyingness for vines}
\label{sec:measure_nonsimplifyingness_vines}

Conditional copulas are the main building blocks of vine models, and the simplifying assumption is of particular importance there.
We refer to \cite{czado2019analyzing} and \cite{czado2022vine} for details on vine models and only present here the corresponding notation.
Formally, a vine $\Vc$ is a sequence of trees $\Tc_1, \dots, \Tc_{d-1}$ such that the edges of $\Tc_k$ become the nodes of $\Tc_{k+1}$ and satisfying the proximity condition.
We denote the node set of $\Tc_k$ by $V_k = V_k(\Vc)$ and the edge set of $\Tc_k$ by $E_k = E_k(\Vc)$.
The vine copula decomposition is the decomposition of the copula density $c_{\X}$ of a continuous random vector $\X$ as
\begin{align*}
    c_{\X}(\x) = \prod_{k = 1}^{d-1} \prod_{e \in E_k}
    c_{a_e, b_e | D_e}
    \big( F_{a_e | D_e}(x_{a_e} | \x_{D_e}) \, , \, 
    F_{b_e | D_e}(x_{b_e} | \x_{D_e})
    \, \big| \, \x_{D_e} \big).
\end{align*}
Therefore, for a given copula $c_\X$ of a random vector $\X$ and for a given vine $\Vc$, we can define the measure of non-simplifyingness of the copula $c_\X$ for the vine structure $\Vc$ by
\begin{align}
    \psi(c_{\X}, \Vc) := \sum_{k = 2}^{d-1}
    \sum_{e \in E_k} \psi(c_{a_e, b_e | D_e}).
\label{eq:vines_sum_measures}
\end{align}
Note that the sum in this measure starts at $d=2$ because the first tree of the vine decomposition is always made up of unconditional copulas; therefore there is no conditioning at these levels.
More generally, we can define a measure of non-simplifyingness of the copula $c_\X$ for the vine structure $\Vc$ by
\begin{align}
    \psi(c_{\X}, \Vc) := \Big\|
    \big(\psi(c_{a_e, b_e | D_e}) \big)_{k = 2, \dots, d-1, \, e \in E_k} \Big\|,
\label{eq:vines_norm_measures}
\end{align}
for any norm $\| \, \cdot \,\|$ on $\Rb^{\sum_{k = 2}^{d-1} \Card(E_k)}$.

\medskip

We now switch to a different goal: finding a criteria that would measure how simplified a copula is, when being decomposed by different vines.
For a dimension $d$, let $\Vscr_d$ denotes the collection of all $d$-dimensional vines.
For a given copula density $c_{\X}$, we define three non-simplifyingness scores.
\begin{itemize}
    \item Worst-case non-simplifyingness score:
    \begin{align}
        \text{WCNS}(c_{\X}) := \max_{\Vc \in \Vscr_d} \psi(c_{\X}, \Vc).
    \label{eq:WCNS}
    \end{align}

    \item Best-case non-simplifyingness score:
    \begin{align}
        \text{BCNS}(c_{\X}) := \min_{\Vc \in \Vscr_d} \psi(c_{\X}, \Vc).
    \label{eq:BCNS}
    \end{align}
    
    \item Average-case non-simplifyingness score:
    \begin{align}
        \text{ACNS}(c_{\X}) := \frac{1}{\Card (\Vscr_d)}
        \sum_{\Vc \in \Vscr_d} \psi(c_{\X}, \Vc).
    \label{eq:ACNS}
    \end{align}
\end{itemize}

\begin{example}
    The Gaussian copula is always simplified, so all these three measures are zero. But in general, they are different.
\end{example}

If a copula has a low worst-case non-simplifyingness score, then it is close to be simplified for all vines structures. Then it does not matter so much which vine structure one take.
The best-case non-simplifyingness score is a more pessimistic measure, as it tells us how much non-simplified the copula has to be whatever vine we choose.
The notion of average-case non-simplifyingness is motivated by the statistical practice: what if we choose a vine structure at random, how non-simplified would it be?

\medskip

To conclude this section, we propose several open problems related to these non-simplifyingness scores.
\begin{enumerate}
    \item For usual copula models, how different can their worst-case and best-case non-simplifyingness scores be?

    \item What is the average-case non-simplifyingness score of a typical copula?

    \item How do non-simplifyingness scores change with the dimension?

    \item Are these non-simplifyingness scores very different when replacing the set $\Vscr_d$ of all vines by particular classes of vines such as the D-vines and C-vines in the definitions above?
\end{enumerate}

\section{Estimation of measures of non-simplifyingness}
\label{sec:statistical_inference}

\subsection{Estimation of measures of non-simplifyingness for conditional copulas}

In practice, true copulas and conditional copulas are typically unknown. Therefore, the corresponding measures of non-simplifyingness are also unknown. Nonetheless, they may be important for statistical estimation: if a copula is far from being simplified, and we have enough data points, the statistician may decide to use non-simplified models. On the contrary, if the copula is barely non-simplified (as can be indicated by a low estimated measure of non-simplifyingness), then a simplified model may be good enough.

\medskip

We now assume that we have an i.i.d. dataset $(\X_i, \Z_i)$, for $i = 1, \dots, n$, following the same distribution as the random vector $(\X, \Z)$.
To estimate measures of non-simplifyingness, the easiest method is to use plug-in estimation: one start by estimating conditional copulas, then they can be substituted in the definition of the measure of non-simplifyingness to get an estimator of it.
For example, the measure of non-simplifyingness
\begin{align*}
    \psi(C_{\X | \Z})
    = \| C_{\X | \Z = \cdot} - C_{\X | \Z , ave} \|
\end{align*}
can be estimated by the plug-in estimator
\begin{align}
    \widehat\psi(C_{\X | \Z})
    = \| \widehat C_{\X | \Z = \cdot}
    - \widehat C_{\X | \Z , ave} \|,
\label{eq:estimated_measure-nonsimplified}
\end{align}
where $\widehat C_{\X | \Z = \cdot}$ and $\widehat C_{\X | \Z , ave}$ are respectively estimators of $C_{\X | \Z = \cdot}$ and $C_{\X | \Z , ave}$.
Several estimators of conditional copulas have been proposed and studied in the literature, see \cite{DerumignyFermanian2017}, Chapter 6.3 in \cite{hofert2018elements} and references therein.
For example, following~\cite{DerumignyFermanian2017},
we can use kernel-based estimators of conditional copulas, defined by
\begin{align*}
    \widehat C_{\X|\Z} (\u| \Z = \z)
    &:= \widehat F_{\X|\Z}
    \Big(
    \widehat F_{X_1|\Z}^{-}(u_1 | \Z = \z), \dots, 
    \widehat F_{X_d|\Z}^{-}(u_d | \Z = \z)
    \, \Big| \, \Z = \z \Big),
\end{align*}
where $h > 0$ is a bandwidth, $K$ is a $p$-dimensional kernel and
\begin{align*}
    \widehat F_{\X|\Z} (\x | \Z = \z)
    &:= \frac{1}{n}
    \sum_{i=1}^n K_n (\Z_i, \z)
    \ii  \X_i \leq \x), \\
    \widehat F_{X_k|J} (x | \Z = \z)
    &:= \frac{\sum_{i=1}^n K_n (\Z_i, \z)
    \ii  X_{i,k} \leq x)}{\sum_{j=1}^n K_n (\Z_i, \z)}, \\
    K_n (\Z_i, \z)
    &:= h^{-p} K \left(
    \frac{\widehat F_{Z_1}(Z_{i,1}) - \widehat F_{Z_1}(z_1)}{h},
    \dots,
    \frac{\widehat F_{Z_p}(Z_{i,p}) - \widehat F_{Z_p}(z_p) }{h} \right).
\end{align*}
The average conditional copula will be estimated by either
\begin{equation}
    \widehat C_{\X|\Z, ave}^{(3)} (\cdot)
    := \int \widehat C_{\X|\Z}(\, \cdot \, | \Z = \z)\,
    \widehat F_\Z(d\z)
    = \frac{1}{n} \sum_{i=1}^n
    \widehat C_{\X|\Z}(\, \cdot \, | \Z = \Z_i).
    \label{estimator_Cs3}
\end{equation}
or
\begin{equation}
    \widehat C_{\X|\Z, ave}^{(4)} (\u_I)
    := \frac{1}{n} \sum_{i=1}^n
    \1 \left(
    \widehat F_{X_1|\Z}(X_{i,1} | \Z_i) \leq u_1 , \dots, 
    \widehat F_{X_d|\Z}(X_{i,d} | \Z_i) \leq u_p \right),
    \label{estimator_Cs4}
\end{equation}

In general, this will give strongly consistent estimators of the measures of non-simplifyingness under weak conditions. The derivation of the consistency in a particular case is detailed in Section~\ref{subsec:ex_CKT}.

\subsection{Estimation of measures of non-simplifyingness based on conditional Kendall's tau}
\label{subsec:ex_CKT}

Following~\cite{derumigny2019kernel}, the conditional Kendall's tau $\tau_{1,2|\Z = \z}$ between $X_1$ and $X_2$ can be estimated by
\begin{align*}
    \widehat\tau_{1,2|\Z = \z}
    := \frac{\sum_{i,j = 1}^n w_{i,n}(\z) w_{j,n}(\z)
    \sign \big((X_{i,1} - X_{j,1})
    (X_{i,2} - X_{j,2}) \big)}{
    1 - \sum_{i=1}^n w_{i,n}^2(\z)} 
\end{align*}
where
$\sign(x) := \1_{\{x > 0\}} - \1_{\{x < 0\}}$,
and $w_{i,n}(\z) := K_h(\Z_i-\z) / \sum_{j=1}^n K_h(\Z_j-\z)$,
with $K_h(\cdot):= h^{-p} K(\cdot/h)$ for some kernel $K$ on $\Rb^p$, and $h=h(n)$ denotes a bandwidth sequence that tends to zero when $n\to\infty$.
There are other estimators of conditional Kendall's tau: \cite{derumigny2020kendall} proposed to use a parametric regression-type model which allows for faster rates of convergence, and \cite{derumigny2019classification} uses machine learning techniques.

\medskip

Let $\Zc$ to be a compact subset of $\Rb^p$ on which the density of $\Z$ is lower bounded by a positive constant.
Then by Theorem 8 of~\cite{derumigny2019kernel}, under some regularity conditions on the kernel $K$ and the joint distribution of $(\X, \Z)$, if $n h_{n}^{2p} / \log n \to \infty$,
then $\sup_{\z \in \Zc}
\big| \widehat\tau_{1,2|\Z = \z} - \tau_{1,2|\Z = \z} \big| \to 0$ almost surely.
Note that, by the triangular inequality, we have
\begin{align*}
    \sup_{\z, \z' \in \Zc}
    \big| \widehat\tau_{1,2|\Z = \z}
    - \widehat\tau_{1,2|\Z = \z'} \big|
    &\leq \sup_{\z, \z' \in \Zc}
    \big| \tau_{1,2|\Z = \z}
    - \tau_{1,2|\Z = \z'} \big| \\
    &+ \sup_{\z, \z' \in \Zc}
    \big| \widehat\tau_{1,2|\Z = \z}
    - \widehat\tau_{1,2|\Z = \z'} 
    - \tau_{1,2|\Z = \z}
    + \tau_{1,2|\Z = \z'} \big|.
\end{align*}
Therefore,
\begin{align*}
    \sup_{\z, \z' \in \Zc}
    \big| \widehat\tau_{1,2|\Z = \z}
    - \widehat\tau_{1,2|\Z = \z'} \big|
    &- \sup_{\z, \z' \in \Zc}
    \big| \tau_{1,2|\Z = \z}
    - \tau_{1,2|\Z = \z'} \big| \\
    &\leq \sup_{\z, \z' \in \Zc}
    \big| \widehat\tau_{1,2|\Z = \z}
    - \widehat\tau_{1,2|\Z = \z'} 
    - \tau_{1,2|\Z = \z}
    + \tau_{1,2|\Z = \z'} \big| \\
    &\leq 2 \sup_{\z \in \Zc}
    \big| \widehat\tau_{1,2|\Z = \z}
    - \tau_{1,2|\Z = \z} \big|,
\end{align*}
by the triangular inequality.
By interchanging $\widehat\tau$ and $\tau$, we obtain that
\begin{align*}
    &\left|
    \sup_{\z, \z' \in \Zc}
    \big| \widehat\tau_{1,2|\Z = \z}
    - \widehat\tau_{1,2|\Z = \z'} \big|
    - \sup_{\z, \z' \in \Zc}
    \big| \tau_{1,2|\Z = \z}
    - \tau_{1,2|\Z = \z'} \big|
    \right|
    \leq 2 \sup_{\z \in \Zc}
    \big| \widehat\tau_{1,2|\Z = \z}
    - \tau_{1,2|\Z = \z} \big|,
\end{align*}
which tends almost surely to $0$.

\medskip

Therefore, we have shown that
\begin{align}
    \widehat\psi
    = \sup_{\z, \z' \in \Zc}
    \big| \widehat\tau_{1,2|\Z = \z}
    - \widehat\tau_{1,2|\Z = \z'} \big|
\label{eq:estimated_measure-CKT-pairwise}
\end{align}
is a strongly consistent estimator of
\begin{align*}
    \psi = \sup_{\z, \z' \in \Zc}
    \big| \tau_{1,2|\Z = \z}
    - \tau_{1,2|\Z = \z'} \big|
\end{align*}
in this setting.

\medskip

In the same way, given a finite set of design points $\z_1, \dots, \z_{n'}$, one can prove that
\begin{align}
    \widehat\psi = \sup_{i, j = 1, \dots, n'}
    \big| \widehat\tau_{1,2|\Z = \z_i}
    - \widehat\tau_{1,2|\Z = \z_j} \big|
\label{eq:estimated_measure-CKT-pairwise-finite}
\end{align}
is a strongly consistent estimator of
\begin{align*}
    \psi = \sup_{i, j = 1, \dots, n'}
    \big| \tau_{1,2|\Z = \z_i}
    - \tau_{1,2|\Z = \z_j} \big|.
\end{align*}
Replacing supremum by sums, we can observe that the same result holds for the sum-type pseudo-measure of non-simplifyingness:
\begin{align}
    \widehat\psi = \sum_{i, j = 1, \dots, n'}
    \big| \widehat\tau_{1,2|\Z = \z_i}
    - \widehat\tau_{1,2|\Z = \z_j} \big|
\label{eq:estimated_measure-CKT-pairwise-finite-sum}
\end{align}
is a strongly consistent estimator of
\begin{align*}
    \psi = \sum_{i, j = 1, \dots, n'}
    \big| \tau_{1,2|\Z = \z_i}
    - \tau_{1,2|\Z = \z_j} \big|.
\end{align*}

\section{Simulation study}
\label{sec:simulation_study}

In this section, we do a small simulation study to evaluate the finite-distance statistical properties of estimators of a few measures of non-simplifyingness.
For simplicity, we fix the sample size to $n = 2000$, $d = 2$, $p = 1$, and focus on non-parametric measures of non-simplifyingness.
Here are the measures that we study:
\begin{itemize}
    \item $\psi_{1, CvM} :=
    \left( \int_{\u, z} |C_{\X | Z = z}(\u)
    - C_{\X | Z , ave}(\u)|^2 d\u dz \right)^{1/2}$;
    
    \item $\psi_{1, KS} :=
    \sup_{\u, z} |C_{\X | Z = z}(\u)
    - C_{\X | Z , ave}(\u)|$;
    
    \item $\tilde \psi_{0, CvM} :=
    \left( \int_{\u, z, z'} |C_{\X | Z = z}(\u)
    - C_{\X | Z = z'}(\u)|^2 d\u dz dz' \right)^{1/2}$;
    
    \item $\tilde \psi_{0, KS} :=
    \sup_{\u, z, z'} |C_{\X | Z = z}(\u)
    - C_{\X | Z = z'}(\u)|$.
\end{itemize}
In all cases, $Z$ is generated following the uniform distribution on $[0, 1]$, the conditional margins of $X_1$ and $X_2$ given $Z = z$ are uniform, and the conditional copula of $(X_1, X_2)$ given $Z = z$ is one of the three choices:
\begin{itemize}
    \item ``\texttt{indep}'': $C_{\X | Z = z}$ is the independence copula;

    \item ``\texttt{gauss\_0\_5}'': $C_{\X | Z = z}$ is the Gaussian copula with correlation $\rho = 0.5$;

    \item ``\texttt{gauss\_0.8z}'': $C_{\X | Z = z}$ is the Gaussian copula with correlation $\rho(z) = 0.8 \times z$.
\end{itemize}

Note that the first two satisfy the simplifying assumption, so their measures of non-simplifyingness are $0$. Actually, for the first case, we even have conditional independence between $X_1$ and $X_2$ given $Z$, as a special case of the simplifying assumption.
The third case does not satisfied the simplifying assumption.
The corresponding measures of non-simplifyingness can then be numerically computed (either by numerical integration, in case of the Cramer-von Mises-type with the $L_2$-norm -- or by numerical optimization, in the case of Kolmogorov-Smirnov-type with the $L_\infty$-norm):
$\psi_{1, CvM} = 0.02383572$;
$\psi_{1, KS} = 0.0680061461$;
$\tilde \psi_{0, CvM} = 0.03194286$;
$\tilde \psi_{0, KS} = 0.1475836177$.

\medskip

We estimate the measures of non-simplifyingness following the framework detailed in the previous section.
These estimation procedures are all available in the \texttt{R} package \texttt{CondCopulas} \cite{derumignyCondCopulas}.
Code to reproduce these experiments is available at the address \url{https://github.com/AlexisDerumigny/Reproducibility-2025-Measures-NonSimplifyingness}.
For $50$ simulations of each model, we estimate the measures of non-simplifyingness for several choices of the bandwidth parameter $h$.
We also have two possibilities of estimating $\psi_{1, CvM}$ and $\psi_{1, KS}$ depending on which estimator of $C_{\X | Z , ave}$ is used (``\texttt{Cs\_3}'' for $\widehat C_{\X|\Z, ave}^{(3)}$, defined in \eqref{estimator_Cs3}, and ``\texttt{Cs\_4}'' for $\widehat C_{\X|\Z, ave}^{(4)}$, defined in \eqref{estimator_Cs4}).
The results are given in Figure~\ref{fig:plot_measure_as_function_h}.

\begin{figure}[htbp]
    \centering
    \includegraphics[width=0.9\linewidth]{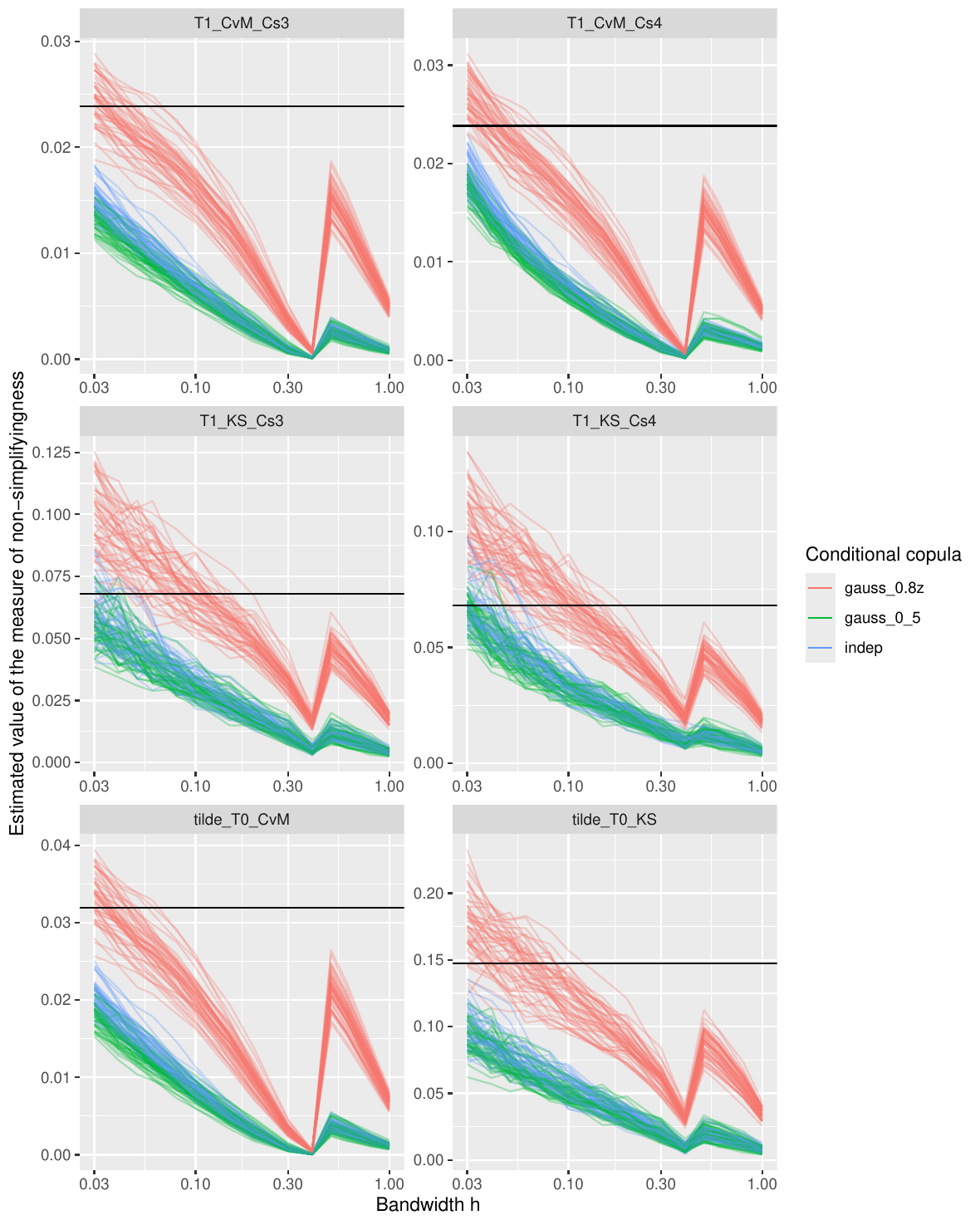}
    \caption{Estimated value of different measures of non-simplifyingness as a function of the bandwidth $h$ for different simulations of 3 data-generating processes. The true value of each measure of non-simplifyingness in the non-simplified case \texttt{gauss\_0.8z} is given by the corresponding horizontal black line.}
    \label{fig:plot_measure_as_function_h}
\end{figure}

\medskip

Overall, it appears that the estimated measures of non-simplifyingness are mostly decreasing with $h$ (for not too large bandwidths), which is to be expected since the smoothing reduces the differences between the conditional copulas when conditioning by different values $\z$.
For the two simplified cases, the estimated measures of non-simplifyingness are quite positively biased (which is to be expected given that they are positive and that the true value is $0$).
Nevertheless, they are smaller than in the non-simplified case, which is reassuring.

\medskip

The type of estimation of $C_{\X | Z , ave}$ does not seem to have a particular influence, but the choice of the norm ($L_2$ compared to $L_\infty$) seems to matter: it looks like better differences are obtained in the $L_2$ case. The curves are also smoother, which is typical of integration compared to taking a supremum.
Interesting, for different measures of non-simplifyingness, the optimal value of the smoothing parameter $h$ seems to differ.


\medskip

\bmhead{Acknowledgements}


The author thanks Claudia Czado, Jean-David Fermanian, Dorota Kurowicka and Thomas Nagler for discussions about this subject.
The author thanks Roger Cooke for mentioning Example~\ref{example:measure_nonconstantness_derivatives} when discussing a draft of this article.
The author thanks an anonymous referee for suggesting Appendix~\ref{sec:list_new_measures}, among other valuable comments that helped to improve the manuscript.

\begin{appendices}




\section{List of new measures}
\label{sec:list_new_measures}

\subsection{Measures and pseudo-measures of non-constantness}

{
\renewcommand{\arraystretch}{1.5}

\begin{tabular}{|l|l|}
    \hline
    Description of the measure & Reference \\
    \hline\hline
    Pseudo-measure of non-constantness induced by a pseudo-norm &
    Example~\ref{ex:measure-non-const_pseudo-norm}
    \\
    \hline
    Discrete measure of non-constantness &
    Example~\ref{ex:measure-non-const_trivial}
    \\
    \hline
    Kolmogorov-Smirnov pseudo-measure of non-constantness &
    Example~\ref{ex:measure-non-const_Kolmogorov-Smirnov}
    \\
    \hline
    Integral-type pseudo-measures of non-constantness &
    Example~\ref{ex:measure-non-const_integral-type}
    \\
    \hline
    Pseudo-measures of non-constantness from averaging &
    Example~\ref{ex:measure-non-const_from-averaging}
    \\
    \hline
    Pseudo-measures of non-constantness from conic combinations &
    Remark~\ref{rem:structure_set_measures_non-constantness}
    \\
    \hline
    Measures of non-constantness from derivatives &
    Example~\ref{example:measure_nonconstantness_derivatives}
    \\
    \hline
\end{tabular}
}

\subsection{Measures and pseudo-measures of non-simplifyingness for conditional copulas}

{
\renewcommand{\arraystretch}{1.5}

\begin{tabular}{|p{0.72\textwidth}|p{0.2\textwidth}|}
    \hline
    Description of the measure & Reference \\
    \hline\hline
    Measure of non-simplifyingness from measure of non-constantness &
    Proposition~\ref{prop:measures_nonsimplifyingness-from-non-constantness}
    \\
    \hline
    Measure of non-simplifyingness based on the distance to the average copula & 
    Equation~\eqref{eq:dist-to-average-copula}
    \\
    \hline
    Measure of non-simplifyingness based on the distance to the average Kendall's tau & 
    Equation~\eqref{eq:dist-to-average-Kendall-tau}
    \\
    \hline
    Measure of non-simplifyingness based on the distance to the average Spearman's rho & 
    Equation~\eqref{eq:dist-to-average-Spearman-rho}
    \\
    \hline
    Measure of non-simplifyingness based on general pairwise comparisons of conditional copulas & 
    Equation~\eqref{eq:comparison-copulas-pairwise}
    \\
    \hline
    Measure of non-simplifyingness based on pairwise comparisons of conditional copulas with the $L_\infty$-norm & 
    Equation~\eqref{eq:comparison-copulas-pairwise-sup}
    \\
    \hline
    Measure of non-simplifyingness based on integrated-type pairwise comparisons of conditional copulas & 
    Equation~\eqref{eq:comparison-copulas-pairwise-int}
    \\
    \hline
    Measure of non-simplifyingness based on a parametrized model & 
    Equations~\eqref{eq:measure-nonsimplified-parametric} and~\eqref{eq:measure-nonsimplified-parametric-pairwise}
    \\
    \hline
    Measure of non-simplifyingness for meta-elliptical copulas & 
    Equation~\eqref{eq:measure-nonsimplified-elliptical}
    \\
    \hline
    Measure of non-simplifyingness with non-continuous marginsfor meta-elliptical copulas & 
    Equations~\eqref{eq:measure-nonsimplified-pairwise} and~\eqref{eq:measure-nonsimplified-strict} 
    \\
    \hline
\end{tabular}
}

\bigskip

\subsection{Measures of non-simplifyingness for vines}

{
\renewcommand{\arraystretch}{1.5}

\begin{tabular}{|p{0.72\textwidth}|p{0.2\textwidth}|}
    \hline
    Description of the measure & Reference \\
    \hline\hline
    Measure of non-simplifyingness of a copula for a given vine structure &
    Equations~\eqref{eq:vines_sum_measures} and~\eqref{eq:vines_norm_measures} 
    \\
    \hline
    Worst-case non-simplifyingness scores of a copula &
    Equation~\eqref{eq:WCNS}
    \\
    \hline
    Best-case non-simplifyingness scores of a copula &
    Equation~\eqref{eq:BCNS}
    \\
    \hline
    Average-case non-simplifyingness scores of a copula &
    Equation~\eqref{eq:ACNS} 
    \\
    \hline
\end{tabular}
}

\subsection{Estimated measures of non-simplifyingness}

{
\renewcommand{\arraystretch}{1.5}

\begin{tabular}{|p{0.72\textwidth}|p{0.2\textwidth}|}
    \hline
    Description of the measure & Reference \\
    \hline\hline
    Estimated measure obtained by plug-in of the conditional copulas &
    Equation~\eqref{eq:estimated_measure-nonsimplified}
    \\
    \hline
    Estimated measure obtained by plug-in of the conditional Kendall's tau &
    Equation~\eqref{eq:estimated_measure-CKT-pairwise}
    \\
    \hline
    Estimated measure obtained by plug-in of a finite number of the conditional Kendall's tau &
    Equations~\eqref{eq:estimated_measure-CKT-pairwise-finite}
    and~\eqref{eq:estimated_measure-CKT-pairwise-finite-sum}
    \\
    \hline
\end{tabular}
}



\end{appendices}

\bigskip


\bibliography{mainV3}

\end{document}